\documentclass[preprint]{elsarticle}

\usepackage{amssymb}
\usepackage{amsmath}
\usepackage{amsthm}
\usepackage{enumerate}
\usepackage{graphicx}
\usepackage{amsfonts}
\usepackage{mathtools,xcolor,url}
 \usepackage{minipage}
\usepackage{float}
\usepackage{todonotes}
\newtheorem{tb}{Table}

\journal{...}

\newtheorem{theorem}{Theorem}[section]
\newtheorem{definition}{Definition}[section]

\newtheorem{cor}{Corollary}[section]
\newtheorem{prop}{Proposition}[section]
\newtheorem{remark}{Remark}[section]

\newtheorem{conclusion}{Conclusion}[section]
\newtheorem{exam}{Example}[section]

\newtheorem{Conjecture}{Conjecture}[section]

\numberwithin{equation}{section}

\begin{document}

	\begin{frontmatter}

		\title{Aldaz-Kounchev-Render operators and their approximation properties }

		\author[1]{Ana-Maria Acu} 
		\author[2,3]{Stefano De Marchi}
		\author[4]{Ioan Ra\c sa}
		
		\vspace{10cm}
		

		\address[1]{Lucian Blaga University of Sibiu, Department of Mathematics and Informatics,  Romania, e-mail: anamaria.acu@ulbsibiu.ro}
		\address[2]{University of Padova, Department of Mathematics “Tullio Levi-Civita”, Italy,
			e-mail:  stefano.demarchi@unipd.it}
		\address[3]{INdAM Gruppo Nazionale di Calcolo Scientifico}
		
		\address[4]{Technical University of Cluj-Napoca, Faculty of Automation and Computer Science, Department of Mathematics, Str. Memorandumului nr. 28, 400114 Cluj-Napoca, Romania
			e-mail:  ioan.rasa@math.utcluj.ro }

		\begin{abstract} 	
			{The approximation properties of the Aldaz-Kounchev-Render (AKR) operators are discussed and classes of functions for which these operators approximate better than the classical Bernstein operators are described. The new results are then extended to the bivariate case on the square $[0,1]^2$ and compared with other existing results known in literature. Several numerical examples, illustrating the relevance and supporting the theoretical findings, are presented.} 
		\end{abstract}

		\begin{keyword} Aldaz-Kounchev-Render operators; Bernstein operator; convex functions; tensor product.
			
			\MSC[2010]  41A36,  41A15, 65D30.
		\end{keyword}
		
	\end{frontmatter}

	\section{Introduction}
	The functions fixed by the positive linear operators $L_n$, $n\geq 1$, encode important information about the operators. Algebraically, because they are eigenfunctions associated with the eigenvalue $1$. Analytically, because they have impact on the approximation properties of $L_n$: in particular, they determine to a large extent the structure of the Voronovskaja operator associated with the sequence $(L_n)_{n\geq 1}$. And, very important for CAGD, the shape preserving properties of the operators are intimately related with the fixed functions (see \cite{AKR2009}).
	
	Generally speaking, given a classical sequence $(L_n)_{n\geq 1}$ the functions fixed by $L_n$ are well-known. In the last decades an inverse problem was raised: given some functions, construct a sequence of positive linear operators $(L_n)_{n\geq 1}$ (an approximation process) such that each $L_n$ fixes the given functions. For example, the operators constructed by King \cite{1} on $C[0,1]$ preserve the functions {\bf 1} and $x^2$. Operators on $C[0,1]$ preserving {\bf 1} and a given function $\tau$ were constructed in \cite{3} and  \cite{Gonska}.   Aldaz,  Kounchev and  Render \cite{AKR2009} introduced a sequence of Bernstein type polynomial operators on $C[0,1]$ which preserve {\bf 1} and $x^j$, $j\in {\mathbb N}$ being given. Several papers were subsequently devoted to their study. In particular, the Voronovskaja formula conjectured in \cite{Rasa2012} was proved in \cite{Bir2017} (see also \cite{Gavrea}, \cite{Finta1}, \cite{Acu}).
	
	The aim this paper is twofold. On one hand, we introduce the bivariate AKR operators on $C([0,1]^2)$ and investigate some of their approximation properties. On the other hand, we compare, in the univariate case and also in the bivariate case, the properties of Bernstein and AKR operators. In particular, we describe classes of functions on which AKR operators approximate better than Bernstein operators, and classes of functions on which the approximation given by Bernstein operators is better that the approximation provided  by AKR operators. The analytic description is accompanied by numerical and graphical experiments.
	
	\begin{definition}
		Let $I,J$ be  compact intervals of the real axis, and let $L:C(I)\to C(I)$, $M:C(J)\to C(J)$ be discretely defined operators
		$$L(f;x)=\displaystyle\sum_{i=0}^n f(x_{i})p_{i}(x),\,\, f\in C(I),  $$
		and
		$$M(f;y)=\displaystyle\sum_{k=0}^m f(y_{k})q_{k}(y),\,\, f\in C(J), $$
		where  $x_{i}\in I$, $y_{k}\in J$ are mutually distinct, and $p_{i}\in C(I)$, $q_{j}\in C(J)$.
	\end{definition}
Let $(x,y)\in I\times J$.  The parametric extensions of $L$ and $M$ to $C(I\times J)$ are given by
	$$_xL(f;x,y)=\displaystyle\sum_{i=0}^nf(x_{i},y)p_{i}(x)$$
	and
	$$_yM(f;x,y)=\displaystyle\sum_{k=0}^mf(x,y_{k})q_{k}(y).$$
	The tensor product of $L$ and $M$ is given by
	$$Tf(x,y):= \left(_xL\circ \,_yM\right)(f;x,y)=\displaystyle\sum_{i=0}^n\sum_{k=0}^m f(x_{i},y_{k})p_{i}(x)q_{k}(y),\,\, f\in C(I\times J). $$
	Throughout the paper we use the notation $e_i(x)=x^i,\,\, i=0,1,\dots$ and 
	$\|\cdot\|$ stands for the supremum norm.

	\section{Application to Bernstein operators}
	Let   $B_{n}:C[0,1]\to C[0,1]$ be the classical Bernstein operator defined as
	$$B_n(f;x)=\displaystyle\sum_{i=0}^n f\left(\dfrac{i}{n}\right)p_{n,i}(x),$$ where   $p_{n,i}(x)={n\choose i}x^i(1-x)^{n-i},\, x\in [0,1]$.
	Then, for $f\in C([0,1]^2)$,  the tensor product of $B_{n}, B_m$ is given by
	\begin{equation}\label{Bernstein} B_{n,m}
		f(x,y):=(\, _xB_n\circ \,_yB_m)f(x,y)=\displaystyle\sum_{i=0}^{n}\sum_{k=0}^{m}f\left(\dfrac{i}{n}, \frac{k}{m}\right)p_{n,i}(x) p_{m,k}(y). \end{equation}

	Denote by $C^{p,q}([0,1]^2)$ the space of all real valued functions defined on $[0,1]^2$ and having continuous partial derivatives of order $p$, respectively $q$.

	Let $f\in C^{2,2}([0,1]^2)$.  The following approximation formula for  the bivariate Bernstein operators $B_{n,m}$  was obtained in \cite{Banach}
	\begin{equation}
		\label{eq.1} \left| f(x,y)-B_{n,m}f(x,y)\right|\leq \dfrac{3}{2}\left[\dfrac{x(1-x)}{n}\|f^{(2,0)}\|+\dfrac{y(1-y)}{m}\|f^{(0,2)}\|\right].
	\end{equation}
	
	Motivated by the  well-known results for the classical Bernstein operators, namely
	\begin{equation}\label{e2} |f(x)-B_{n}(f;x)|\leq\displaystyle\frac{1}{2}\| f^{\prime\prime}\|\frac{x(1-x)}{n},  \end{equation}
	in \cite{Banach} the constant $3/2$  was improved and the following approximation formula was obtained (see also \cite[Theorem 2.3]{BarPop2008})
	\begin{align}
		\label{eq.2} 
		&\left| f(x,y)-B_{n,m}f(x,y)\right|\\
		&\leq \dfrac{x(1-x)}{2n}\|f^{(2,0)}\|+\dfrac{y(1-y)}{2m}\|f^{(0,2)}\|+\dfrac{x(1-x)y(1-y)}{4nm}\|f^{(2,2)}\|.\nonumber
	\end{align}
	
	In the following we obtain a new result that improves  (\ref{eq.1}) and (\ref{eq.2}). 
	
	\begin{prop} 
		Let $f\in C^{2,2}([0,1]^2)$. Then
		\begin{equation}
			\label{eq.3} \left| f(x,y)-B_{n,m}f(x,y)\right|\leq \dfrac{1}{2}\left[\dfrac{x(1-x)}{n}\|f^{(2,0)}\|+\dfrac{y(1-y)}{m}\|f^{(0,2)}\|\right].
		\end{equation}
		
	\end{prop}
	
	\begin{proof}
		Using the estimate (\ref{e2}) we get
		\begin{align*}
			&|f(x,y)-B_{n,m}f(x,y)|\leq |f(x,y)-\,_xB_{n}f(x,y)|\\
			&+|_xB_{n}f(x,y)-(\, _xB_n\circ \,_yB_m)f(x,y)|\\
			&\leq |f(x,y)-\,_xB_{n}f(x,y)|+\,_xB_n\left(|f-\,_yB_mf|\right)(x,y)\\
			&\leq\dfrac{x(1-x)}{2n}\|f^{(2,0)}+\,_xB_n\left(\dfrac{y(1-y)}{2m}\|f^{(0,2)}\|\right)\\
			&\leq \dfrac{1}{2}\left[\dfrac{x(1-x)}{n}\|f^{(2,0)}\|+\dfrac{y(1-y)}{m}\|f^{(0,2)}\|\right].
		\end{align*}
	\end{proof}

	\section{Bernstein-type operator of Aldaz, Kounchev and Render}
	
	Starting from the classical Bernstein operators $B_n$ defined on $ C[0,1]$,
	during recent years some modifications have been considered. One of them was introduced by J.P. King \cite{1} in order to obtain linear positive operators which preserve  the functions $e_0$ and $e_2$.  A slight extension of King operators was considered by C\'ardenas-Morales et al. in \cite{CGM2006} 
	where  a sequence of operators $B_{n,\alpha}$ that preserve $e_0$ and $e_2+\alpha e_1$, $\alpha\in [0,+\infty)$ was introduced.
	Using a continuous strictly increasing function $\tau$ defined on $[0,1]$ with $\tau(0)=0$ and $\tau(1)=1$, $\tau'(x)>0, x\in  [0,1]$, C\'ardenas-Morales et al. (see  \cite{3}, \cite{Gonska}) introduced a modification of the Bernstein operator which preserves the functions $e_0$ and $\tau$.

	For $j>1$, $j\in {\mathbb N}$  fixed and   $n\geq j$,  Aldaz, Kounchev and Render \cite{AKR2009} introduced a { polynomial}  operator $B_{n,j}: C[0,1]\to C[0,1]$ that fixes $e_0$ and $e_j$.  The operator  is explicitly given by
	\begin{equation}\label{e4} B_{n,j}(f;x)=\displaystyle\sum_{k=0}^n f\left(t_{n,k}^j\right)p_{n,k}(x),\end{equation}
	where $$t_{n,k}^j=\left(\dfrac{k(k-1)\dots(k-j+1)}{n(n-1)\dots(n-j+1)}\right)^{1/j}.$$
	
	For $n = 10$ and $j = 2$ in Figure \ref{fig:4} the nodes of  AKR operator and Bernstein operator, respectively, are illustrated graphically. 
	
	\begin{figure}[htbp]
		\centering
		\includegraphics[height=30mm,keepaspectratio]{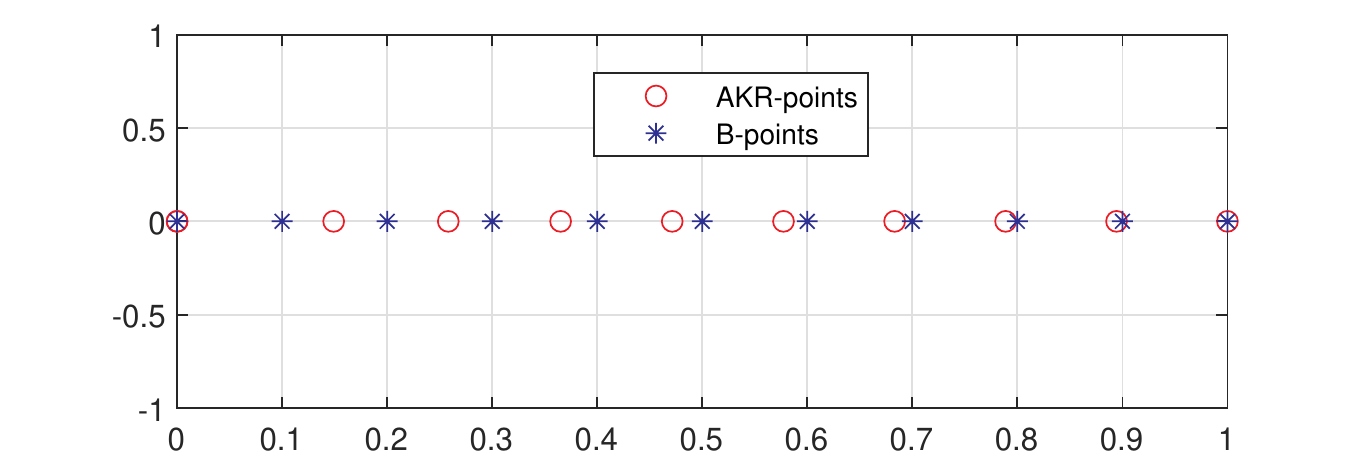}
		\caption{Nodes of $B_{n,j}$ and $B_n$ } \label{fig:4}
	\end{figure}
	
	Next we determine a class of functions for which the approximation by  the AKR operators is better than the approximation by the Bernstein operators. In order to describe  this class of functions, in the sequel  we will  recall some necessary  notions from the literature (see \cite{AKR2009}).

	We say that $(f_0,f_1)$ is a {\it Haar pair} if $f_0$ is strictly positive and $f_1/f_0$ strictly increasing.
	
	\begin{definition}(see \cite[p. 280]{18})
		A function $f:[a,b]\to {\mathbb R}$ is called $(f_0,f_1)$-convex if for all $x_0,x_1,x_2\in [a,b]$ with $x_0<x_1<x_2$, the determinant
		$$ det\left(\begin{array}{ccc}
			f_0(x_0) & f_0(x_1) &f_0(x_2)\\
			f_1(x_0) & f_1(x_1) &f_1(x_2)\\
			f(x_0) & f(x_1) &f(x_2)\\
		\end{array}\right)  $$
		is non-negative.
		
	\end{definition}
	
	We will use the following characterization of $(f_0,f_1)$-convexity given in \cite[Theorem 5]{5}:
	
	\begin{theorem}(\cite{5})\label{t4.1}
		Let $(f_0,f_1)$ be a Haar pair and let $I:=(f_1/f_0)([a,b])$. Then $f\in C[a,b]$ is $(f_0,f_1)$-convex if and only if $(f/f_0)\circ (f_1/f_0)^{-1}\in C(I)$ is convex in the standard sense.
	\end{theorem}
	Let ${\cal B}_n$ be a Bernstein type operator defined as
	$$ {\cal B}_nf(x)=\displaystyle\sum_{k=0}^n f(t_{n,k})\alpha_{n,k}p_{n,k} (x),\, f\in C[a,b], x\in [a,b], $$
	where $\alpha_{n,i}>0, i=1,\dots,n$, $t_{n,k}\in [a,b]$.
	
	The next result generalizes the well known inequality verified by the classical Bernstein operator $$B_nf\geq f$$ for all convex functions $f\in C[0,1]$.
	
	\begin{theorem}(\cite[Th.15]{AKR2009})\label{t4.2} Assume that for some $n\geq 1$, there is a Bernstein type operator ${\cal B}_n$ fixing $f_0$ and $f_1$. Then for every $(f_0,f_1)$-convex function $f\in C[a,b]$ we have ${\cal B}_nf\geq f$.
	\end{theorem}
	
	In what follows let $j\geq 2$ and 
	$$ K_j^{[1]}:=\left\{ f\in C[0,1]\, |\, f\textrm{ is increasing, } g(x):=f(x^{1/j}) \textrm{ is convex on } [0,1]\right\}. $$		
	Let $\Omega$ be the set of the functions $\omega$ such that
	\begin{itemize}
		\item[(i)] $\omega\in C[0,1]\cap C^1(0,1]$,
		\item[(ii)] $\omega(x)\geq 0,\,\, x\in[0,1]$,
		\item[(iii)] $\omega^{\prime}(x)\geq 0,\,\, x\in(0,1]$,
		\item[(iv)] there exists $\displaystyle\lim_{x\to 0} x^{j-1}\omega^{\prime}(x)\in\mathbb{R}$.
	\end{itemize}
	
	\begin{theorem}\label{T3.3}
		The following statements are equivalent
		\begin{itemize}
			\item[(a)] $f\in C^2[0,1]\cap K_j^{[1]}$,
			\item[(b)] $f\in C^2[0,1]$, $f^{\prime}(x)\geq 0$, $xf^{\prime\prime}(x)-(j-1)f^{\prime}(x)\geq 0,\, x\in[0,1]$,
			\item[(c)] There exists $\varphi\in \Omega$ such that $f(x)=f(0)+\displaystyle\int_0^x t^{j-1}\varphi(t) dt, \, x\in[0,1]$. 
		\end{itemize}
	\end{theorem}
	
	\begin{proof}
		To prove that (a) and (b) are equivalent let $f\in C^2[0,1]$ and $g(x):=f(x^{1/j})$, $x\in[0,1]$. Then $f\in K_j^{[1]}$ if and only if $f^{\prime}(x)\geq 0$ and $g^{\prime\prime}(x)\geq 0$, $x\in(0,1]$. But $g^{\prime\prime}(x)\geq 0$ for $0<x\leq 1$ is equivalent to $xf^{\prime\prime}(x)-(j-1)f^{\prime}(x)\geq 0,\, x\in[0,1]$.
		
		In order to prove that (b) implies (c) let
		$$f\in C^2[0,1],\, f^{\prime}(x)\geq 0,\, xf^{\prime\prime}(x)-(j-1)f^{\prime}(x)\geq 0,\, x\in[0,1].$$
		Set $\varphi(x)=\dfrac{f^{\prime}(x)}{x^{j-1}}$, $0<x\leq 1$. Then for $x\in(0,1]$ we have
		$$\varphi^{\prime}(x)=\dfrac{xf^{\prime\prime}(x)-(j-1)f^{\prime}(x)}{x^j}\geq 0.  $$
		Moreover, $\varphi\geq 0$ and $\varphi$ is increasing on $(0,1]$. This implies the existence of $l:= \displaystyle\lim_{x\to 0} \varphi(x)\in[0,\infty)$.
		
		Define
		$$ \varphi(x):=\left\{\begin{array}{l}
			\dfrac{f^{\prime}(x)}{x^{j-1}},\,\, x\in(0,1],\\
			\vspace{-0.3cm}\\
			l,\,\, x=0.\end{array}
		\right. $$
		Then $\varphi\in C[0,1]\cap C^1(0,1]$.
		
		From (b) we see that $f^{\prime}(0)=0$. Now $\displaystyle\lim_{x\to 0} \dfrac{f^{\prime}(x)}{x}=f^{\prime\prime}(0)$ and consequently
		$$ \displaystyle\lim_{x\to 0} x^{j-1}\varphi^{\prime}(x)=(1-j)f^{\prime\prime}(0)\in{\mathbb R}. $$
		From $f^{\prime}(t)=t^{j-1}\varphi(t),\,\, t\in[0,1]$, we get
		$f(x)=f(0)+\displaystyle\int_0^x t^{j-1}\varphi(t) dt, \, x\in[0,1]$, and (c) is proved.
		
		It remains to prove that (c) implies (b).
		Let 	$f(x)=f(0)+\displaystyle\int_0^x t^{j-1}\varphi(t) dt$,  $x\in[0,1]$, with $\varphi\in \Omega$. Then $f^{\prime}(x)=x^{j-1}\varphi(x)$,  $x\in[0,1]$. The function $f^{\prime}$ is continuous on $[0,1]$ and 
		\begin{equation}
			\label{eA**} f^{\prime\prime}(x)=(j-1)x^{j-2}\varphi(x)+x^{j-1}\varphi^{\prime}(x) \textrm{ on } (0,1].
		\end{equation}
		Using (\ref{eA**}) and (iv) we infer that there exists $f^{\prime\prime}(0):=\displaystyle\lim_{x\to 0}f^{\prime\prime}(x)\in {\mathbb R}$. Therefore $f\in C^2[0,1]$. From $f^{\prime}(x)=x^{j-1}\varphi(x)$ we see that $f^{\prime}(x)\geq 0, x\in[0,1]$. Using (\ref{eA**}) we get  
		$$xf^{\prime\prime}(x)-(j-1)f^{\prime}(x)\geq 0,\, x\in(0,1].$$
		By continuity this is true also for $x=0$, and now (b) is proved.
		
	\end{proof}
	
	\begin{prop}\label{p4.1}
		Let $f\in K_j^{[1]}$. Then
	\begin{equation}\label{eD1} f\leq B_{n,j}f\leq B_nf.  \end{equation}
	\end{prop}
	\begin{proof}
		It is easy to verify that
		$$t_{n,k}^j\leq \dfrac{k}{n},\,\, k=0,\dots,n,\,\, j\geq 2.   $$
		Since $f$ is increasing, we get $f(t_{n,k}^j)\leq f\left(\dfrac{k}{n}\right)$, therefore
		\begin{equation}\label{e3}
			B_{n,j}f\leq B_nf.
		\end{equation}
		We consider $f_0=e_0$ and $f_1=e_j$ in Theorem \ref{t4.1}. Then $(f_0,f_1)$ is a Haar pair and $\left(f/f_0\right)\circ \left(f_1/f_0\right)^{-1}(x)=f\left(x^{1/j}\right)$ is convex.  Applying Theorem \ref{t4.2} (for ${\cal B}_n=B_{n,j}$) we get
		$B_{n,j}f\geq f$. Combined with  (\ref{e3}) this proves  the inequality (\ref{eD1}).
	\end{proof}

	\begin{cor}\label{p4.2}
		Let $f\in C^2[0,1]$. If
		\begin{equation}\label{e5p}f^{\prime}(x)\geq 0 \textrm{ and } xf^{\prime\prime}(x)-(j-1)f^{\prime}(x)\geq 0,\,\, 0\leq x\leq 1, \end{equation}
		then
		$$ f\leq B_{n,j}f\leq B_nf.  $$
	\end{cor}
	\begin{proof}
		According to (\ref{e5p}) and Theorem \ref{T3.3}, $f\in K_j^{[1]}$ and now is sufficient to apply  Proposition \ref{p4.1}.
	\end{proof}

\begin{remark}\label{r3.1}
	Recall the Voronovskaja formula for $B_{n,j}$. If $x\in(0,1]$ and there exists the finite $f^{\prime\prime}(x)$, then
	\begin{equation}\label{eF1} \displaystyle\lim_{n\to\infty} n\left(B_{n,j}(f;x)-f(x)\right)=\dfrac{1-x}{2}\left[xf^{\prime\prime}(x)-(j-1)f^{\prime}(x)\right]. \end{equation}
	 Combining (\ref{eF1}) and Corollary \ref{p4.2} we see that for $f\in C^2[0,1]$ the statements 
	 $$xf^{\prime\prime}(x)-(j-1)f^{\prime}(x)\geq 0, \, x\in[0,1],
	 		 $$
	 		 and
	 		 $$ B_{n,j}(f;x)\geq f(x), \,x\in[0,1], $$
	 		 are equivalent.
\end{remark}
	
	\begin{cor}\label{c3.2}
		Let $\varphi\in C^1[0,1]$ such that $\varphi(x)\geq 0$, $\varphi^{\prime}(x)\geq 0$, $x\in[0,1]$, and let $f(x)=\displaystyle\int_0^x t^{j-1}\varphi(t) dt$, $x\in[0,1]$. Then 
		$$ f\leq B_{n,j}f\leq B_nf.  $$
	\end{cor}
\begin{proof}
	Let us remark that $\varphi\in \Omega$. Now Theorem \ref{T3.3} shows that $f\in K_j^{[1]}$ and an application of Proposition \ref{p4.1} concludes the proof.
\end{proof}

	\begin{exam}\label{ex.1}
		Let $\varphi(t)=\sin\dfrac{\pi t}{2}$, $t\in[0,1]$. Using Corollary \ref{c3.2} with $j=2$, we determine the function $f(x)=\displaystyle\int_0^x t\varphi(t) dt=-\dfrac{2}{\pi}x\cos\dfrac{\pi x}{2}+\dfrac{4}{\pi^2}\sin\dfrac{\pi x}{2}.$ In Figure \ref{fig:1} it can be seen that for this function the  approximation by  the operator $B_{5,2}$ is better than the approximation by the Bernstein operator $B_5$. Moreover, if we denote
		$$ E_{AKR}(f;n,j)=\left\| B_{n,j} f-f\right\| \textrm{ and } E_{B}(f;n)=\left\| B_{n} f-f\right\|,  $$
		the approximation error by AKR operator and the Bernstein operator,
		in Table \ref{table2} we present  $E_{AKR}f$ and  $E_Bf$ for certain values of $n$.

	\begin{tb}\label{table2}
		\centering
		{\it  Error of approximation }
		
		$  $
		
		\begin{tabular}{cccccccc}\hline
			$n$&$5$&$10$&$20$&$30$&$40$&$50$&$60$ \\  \hline
			$E_{B}f$   & $0.0140$ &  $0.0070$ &  $0.0035$ & $0.0023 $ & $0.0017$ & $0.0014$ & $0.0012$\\
			$E_{AKR}f $   & $0.0309$ &  $0.0159$ &  $0.0081$ & $0.0054$ & $0.0041$ & $0.0033$ & $0.0027$\\
			\hline
		\end{tabular}
	\end{tb}
		
		\medskip
		
		\begin{figure}[htbp]
			\centering
			\includegraphics[height=50mm,keepaspectratio]{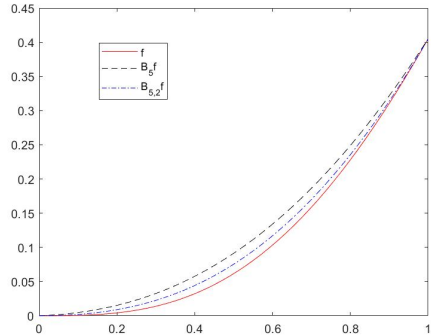}
			\caption{Plots of $f$, $B_{5,2}f$ and $B_5f$} \label{fig:1}
		\end{figure}
	\end{exam}
\begin{exam}
	Take $\varphi(t)=e^t$, $t\in[0,1]$. For $j=5$ Corollary \ref{c3.2} gives the function $f(x)=\displaystyle\int_0^x t^4e^tdt=-24 + (x^4 - 4x^3 + 12x^2 - 24x + 24)e^x$. Figure \ref{F3.2} shows that AKR operator approximates the function  better than the Bernstein  operator. 
	
	\begin{figure}[htbp]
		\centering
		\includegraphics[height=50mm,keepaspectratio]{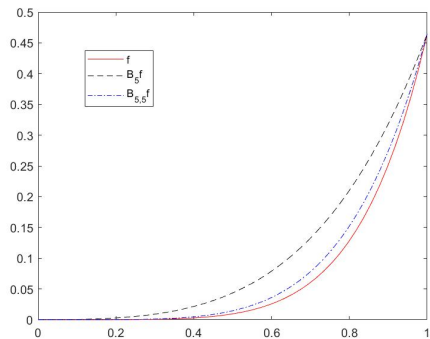}
		\caption{Graph  of $f$, $B_{5,5}f$ and $B_5f$} \label{F3.2}
	\end{figure}

\end{exam}
	
	In the next result we mention a class of functions  for which  the approximation by the Bernstein operators is better than the approximation by AKR operators. 
	
	\begin{prop} Let $f\in C[0,1]$ be a decreasing and  convex function. Then
		\begin{equation}
			\label{eq.A} B_{n,j}f\geq B_n f\geq f.  \end{equation}
	\end{prop}
	\begin{proof} Since the function $f$ is decreasing and $t_{n,k}^j\leq \dfrac{k}{n}$, $k=0,\dots,n$, $j\geq 2$, we get 
		$ f\left(t_{n,k}^j\right)\geq f\left(\dfrac{k}{n}\right)$,
		therefore
		$ B_{n,j}f\geq B_nf$. But, it is  well known that $B_nf\geq f$ for each convex function $f$.
		This leads to the inequalities (\ref{eq.A}). 
	\end{proof}
	
	\begin{exam}
		Let $f(x)=\cos^2\left(\dfrac{\pi}{4}(x+1)\right)$, $x\in[0,1]$. Note that $f$ is a decreasing and convex function and $f\in C[0,1]$. In Figure \ref{fig:2} are represented graphically the functions $f$, $B_{n,j}f$ and $B_n f$ for $n=10$ and $j=2$.  Note that in this case the approximation by the Bernstein operator is better than the approximation by AKR operator. If we consider the notation introduced in Example \ref{ex.1} for the approximation error we have
				$$E_{AKR}(f;10,2)=0.0450,\,\,E_B(f;10)=0.0118. $$
		\begin{figure}[htbp]
			\centering
			\includegraphics[height=50mm,keepaspectratio]{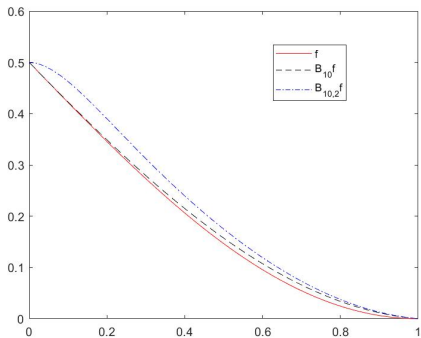}
			\caption{Graph of $f$, $B_{10,2}f$ and $B_{10}f$} \label{fig:2}
		\end{figure}
	\end{exam}
	\section{Bivariate Bernstein-type operator of Aldaz, Kounchev and Render}
	
	Let   $B_{n,j}, B_{m,j}:C[0,1]\to C[0,1]$ be the AKR operators  and $(x,y)\in [0,1]^2$. Then, for $f\in C([0,1]^2)$,  the tensor product of AKR operators is given by
	\begin{equation}\label{Aldaz}B_{n,m,j}(f;x,y)=\displaystyle\sum_{i=0}^{n}\sum_{k=0}^{m}f\left(t_{n,i}^j,t_{m,k}^j\right)p_{n,i}(x) p_{m,k}(y),  \,\, (x,y)\in [0,1]^2.\end{equation}
	
	Since $B_{n,j}$ preserves the functions $1$, $x^j$ and $B_{m,j}$ preserves the functions $1$, $y^j$ defined on $[0,1]$, it follows immediately that $B_{n,m,j}$ preserves the functions $1$, $x^j$, $y^j$ defined on $[0,1]^2$.
	
	The nodes of the operators $B_{10,10,2}$ and $B_{10,10}$, respectively, are presented graphically in Figure \ref{fig:5}. 
	
	\begin{figure}[htbp]
		\centering
		\includegraphics[height=50mm,keepaspectratio]{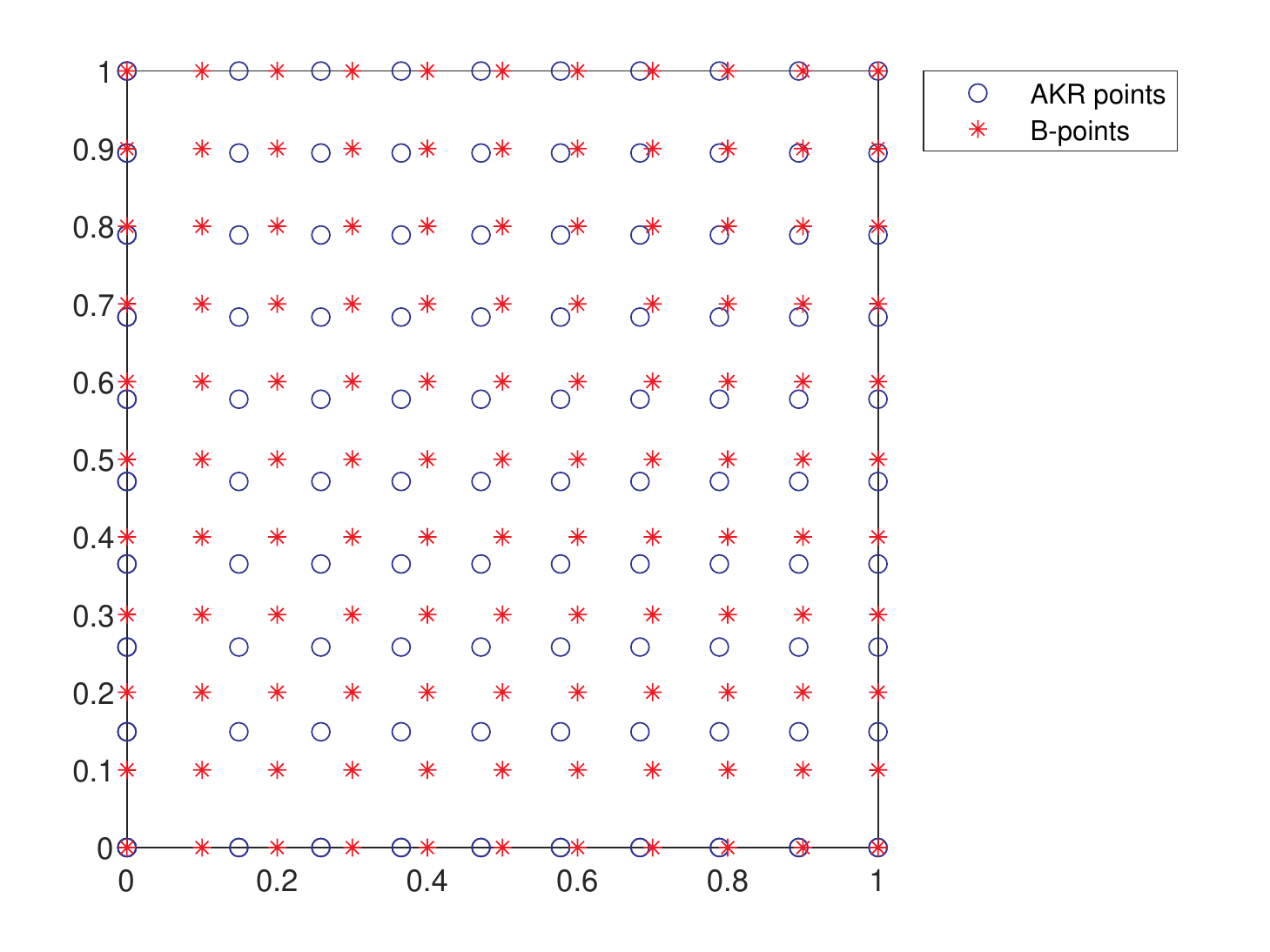}
		\caption{Nodes of $B_{10, 10,2}$ and $B_{10, 10}$ } \label{fig:5}
	\end{figure}
	
	\begin{prop} 
		Let $f\in C^{2,2}([0,1]^2)$. Then
		\begin{align*}
			\left|f(x,y)-B_{n,m,j}(f;x,y)\right|&\leq \dfrac{x(1-x)}{2n}\|f^{(2,0)}\|+\dfrac{y(1-y)}{2m}\|f^{(0,2)}\|\\
			&+\dfrac{j-1}{n}\|f^{(1,0)}\|+\dfrac{j-1}{m}\|f^{(0,1)}\|.
		\end{align*}
		
	\end{prop}
	
	\begin{proof} In \cite{AcuRasa} the following estimate of the difference between the AKR and Bernstein operators was obtained
		\begin{equation}\label{e5}\left| B_n(f;x)-B_{n,j}(f;x)\right|\leq \omega_1\left(f,\dfrac{j-1}{n}\right)\leq \dfrac{j-1}{n}\|f^{\prime}\|,\,f\in C^1[0,1].  \end{equation}
		From (\ref{e2}) and (\ref{e5}) we obtain
		\begin{align}
			|f(x)-B_{n,j}(f;x)|&\leq |f(x)-B_n(f;x)|+|B_n(f;x)-B_{n,j}(f;x)|\nonumber\\
			&\leq \dfrac{x(1-x)}{2n}\|f^{\prime\prime}\|+\dfrac{j-1}{n}\|f^{\prime}\|.\label{eqTT}
		\end{align}
		Using  (\ref{eqTT})  it follows that
		\begin{align*}			&|f(x,y)-B_{n,m,j}(f;x,y)|\leq |f(x,y)-\,_xB_{n,j}(f;x,y)|\\
			&			+|_xB_{n,j}(f;x,y)-\,_xB_{n,j}\circ\, _yB_{m,j}(f;x,y)|\\
			&\leq |f(x,y)-\,_xB_{n,j}(f;x,y)|+\,_xB_{n,j}\left( |f-\, _yB_{m,j}f|\right)(x,y)\\
			&\leq \dfrac{x(1-x)}{2n}\|f^{(2,0)}\|+\dfrac{j-1}{n}\|f^{(1,0)}\|+\dfrac{y(1-y)}{2m}\|f^{(0,2)}\|+\dfrac{j-1}{m}\|f^{(0,1)}\|.
		\end{align*}
		This concludes the proof.
	\end{proof}
	
	\begin{Conjecture}\label{Conj4.1} In relation with Voronovskaja formula for $B_{n,j}$ (see Remark \ref{r3.1}) we state here a conjecture about the Voronovskaja formula for $B_{n,n,j}$. Suppose that $(x,y)\ne \left\{(0,1),(0,1), (1,0)\right\}$, $f\in C([0,1]^2)$ and the partial derivatives $f^{\prime\prime}_{x^2}$, $f^{\prime\prime}_{y^2}$ exist and are finite at $(x,y)$. Then,
	$$
	    \displaystyle\lim_{n\to\infty} n\left(B_{n,n,j}f(x,y)-f(x,y)\right)=Uf(x,y)+Vf(x,y),$$
	where
	$$Uf(x,y):=
	\dfrac{x(1-x)}{2}{f^{\prime\prime}_{x^2}}^2(x,y)	    -\dfrac{j-1}{2}(1-x)f_x^{\prime}(x,y)$$
	  and
	  $$Vf(x,y):=\dfrac{y(1-y)}{2}{f^{\prime\prime}_{y^2}}^2(x,y)
	    -\dfrac{j-1}{2}(1-y)f_y^{\prime}(x,y).$$
	  
		\end{Conjecture}
	Define
	$$ K_j^{[2]}:=\left\{f\in C([0,1]^2)\,|\, f(\cdot, y)\in K_j^{[1]}, \, f(x,\cdot)\in K_j^{[1]},\,x,y\in[0,1]\right\}.$$
	
	\begin{theorem}\label{T4.1}
		If $f\in K_j^{[2]}$, then 
		\begin{equation}\label{eB1} f\leq B_{n,m,j}f\leq B_{n,m}f. \end{equation}
	\end{theorem}
	\begin{proof}
		Let $f\in K_j^{[2]}$. Then $f(\cdot,y)\in K_j^{[1]}$, $y\in[0,1]$. Using Proposition \ref{p4.1} we obtain
		\begin{equation}\label{eB2}f(x,y)\leq \,_xB_{n,j}f(x,y),\, x,y\in[0,1].  \end{equation}
		From (\ref{eB2}) it follows that 
		\begin{equation}\label{eB3} _yB_{m,j}f(x,y)\leq B_{n,m,j}f(x,y). \end{equation}
		On the other hand $f(x,\cdot)\in K_j^{[1]}$ and the same Proposition \ref{p4.1} tells us that 
		\begin{equation}\label{eB4}f(x,y)\leq \,_yB_{m,j}f(x,y).\end{equation}
		Combining (\ref{eB3}) and (\ref{eB4}) we get
		\begin{equation}\label{eB5}
			f(x,y)\leq B_{n,m,j}f(x,y),\,\, x,y\in[0,1].
		\end{equation}
		Remember that $f(x,\cdot)$ and $f(\cdot,y)$ are increasing functions for all $x,y\in[0,1]$. Moreover,
		$$t_{n,i}^j\leq\dfrac{i}{n},\,\,\,\,t_{m,k}^j\leq\dfrac{k}{m}.  $$
		It follows that
		\begin{equation}\label{eB6}f\left(t_{n,i}^j,t_{m,k}^j\right)\leq f\left(\dfrac{i}{n},t_{m,k}^j\right)\leq f\left(\dfrac{i}{n},\dfrac{k}{m}\right),\end{equation}
		for all $i=0,\dots, n$, $k=0,\dots, m. $
		
		From (\ref{Bernstein}),  (\ref{Aldaz}) and (\ref{eB6}) we get
		\begin{equation}\label{eB7p}
			B_{n,m,j}f(x,y)\leq B_{n,m}f(x,y),\,\, x,y\in[0,1].
		\end{equation}
		Now (\ref{eB1}) is a consequence of (\ref{eB5}) and (\ref{eB7p}).
	\end{proof}
	
	\begin{remark}
	Let $f\in C^{2,2}([0,1]^2)$. According to Theorem \ref{T3.3}   if $Uf\geq 0$ and $Vf\geq 0$, then $f\in K_j^{[2]}$ and consequently  Theorem \ref{T4.1} shows that $B_{n,n,j}f\geq f$, $n\geq 1$.
	
	If Conjecture \ref{Conj4.1} is valid and $B_{n,n,j}f\geq f$, $n\geq 1$, then $Uf+Vf\geq 0$.
	
{\bf	Is it true that if $Uf+Vf\geq 0$, then $B_{n,n,j}f\geq f$, $n\geq 1$? }
	\end{remark}
	
	As functions from $K_j^{[1]}$,  $f(\cdot,y)$ and $f(x,\cdot)$ are characterized in Theorem \ref{T3.3}. For the sake of simplicity we will use here the construction described in Theorem \ref{T3.3} with $\varphi\in C^1[0,1]$, $\varphi(x)\geq 0$, $\varphi^{\prime}(x)\geq 0$, $x\in[0,1]$. Suppose that $f\in C^{2,2}([0,1]^2)$. Using convenient notation we have \begin{equation}\label{eB7}f(x,y)=f(0,y)+\displaystyle\int_0^x t^{j-1}\varphi(t,y)dt\end{equation}
	and
	\begin{equation}\label{eB8} f(x,y)=f(x,0)+\displaystyle\int_0^y s^{j-1}\psi(x,s)ds, \end{equation}
	where 
	\begin{align}\label{eB12}&\varphi\in C^{1,1}([0,1]^2),\,\varphi\geq 0,\,\, \varphi_x^{\prime}\geq 0,\\
		& \psi\in C^{1,1}([0,1]^2),\,\psi\geq 0,\,\, \psi_y^{\prime}\geq 0.\label{eB13}\end{align}
	
	Due to (\ref{eB7}) and  (\ref{eB8}) we need the compatibility condition  
	\begin{equation}\label{eB9} f(x,y)=f(0,y)+\displaystyle\int_0^x t^{j-1}\varphi(t,y)dt=f(x,0)+\displaystyle\int_0^y s^{j-1}\psi(x,s)ds.\end{equation}
	Taking in (\ref{eB9}) the derivative with respect to $x$ and then with respect to $y$ we get
	\begin{equation}\label{eB10}t^{j-1}\varphi_y^{\prime}(t,s)=s^{j-1}\psi_x^{\prime}(t,s). \end{equation}
	Conversely, we will show that if (\ref{eB10}) is fulfilled then (\ref{eB7}) and  (\ref{eB8}) give us a function $f(x,y)$ for which the compatibility condition (\ref{eB9}) is fulfilled.
	
	In (\ref{eB10}) take the integral with respect to $t$ on the interval $[0,x]$ and then the integral with respect to $s$ on the interval $[0,y]$. This yields 
	\begin{align}
		&\displaystyle\int_0^x t^{j-1}\varphi(t,y) dt+\int_0^y s^{j-1}\psi(0,s)ds\nonumber\\
		&=\displaystyle\int_0^y s^{j-1}\psi(x,s) ds+\int_0^x t^{j-1}\varphi(t,0)dt=:f(x,y).\label{eB11}
	\end{align}
	It is easy to check that the above function $f(x,y)$ satisfies (\ref{eB9}).
	
	To resume, we have proved 
	\begin{theorem}
		\label{T4.2} Let $\varphi$ and $\psi$ satisfying (\ref{eB12}), (\ref{eB13}) and (\ref{eB10}). The function $f$ given by (\ref{eB11}) is in $K_j^{[2]}$ and satisfies
		$$f\leq B_{n,m,j}f\leq B_{n,m}f. $$
	\end{theorem}

	\begin{exam}
		Let $\varphi(x,y)=\psi(x,y)=h(x^j+y^j), \,h\in C^1[0,2]$, $h\geq 0$, $h^{\prime}\geq 0$. We have $x^{j-1}\varphi_y^{\prime}(x,y)=y^{j-1}\psi_x^{\prime}(x,y)$, and conditions (\ref{eB12}), (\ref{eB13}), (\ref{eB10}) are verified.

The function $f$ given by (\ref{eB11}) is in this case
$$ f(x,y)=\dfrac{1}{j}\int_0^{x^j+y^j}h(u)du. $$
	\end{exam}
	
	\begin{exam}\label{Ex4}  Let  $\varphi(x,y)=2y^2e^{x^2y^2}$ and $\psi(x,y)=2x^2e^{x^2y^2}$. For $j=2$ the conditions (\ref{eB12}), (\ref{eB13}), (\ref{eB10}) are verified. Therefore, from (\ref{eB11}) we get 
		$$f(x,y)=\displaystyle\int_0^x t^{j-1}\varphi(t,y)dt =e^{x^2y^2}-1.$$
		For $n=3$, $m=4$ and $j=2$  Figures \ref{F4.1.1} and \ref{F4.1.2} illustrate the inequalities\linebreak
		$f\leq B_{3,4,2}f\leq B_{3,4}$ (see Theorem \ref{T4.2}).
		
			\begin{minipage}{\linewidth}
			\centering
			\begin{minipage}{0.45\linewidth}
				\begin{figure}[H]
					\includegraphics[width=\linewidth]{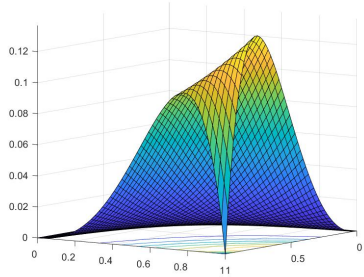}
					\caption{Graph of $B_{3,4,2}f-f$}\label{F4.1.1}
				\end{figure}
			\end{minipage}
			\hspace{0.05\linewidth}
			\begin{minipage}{0.45\linewidth}
				\begin{figure}[H]
					\includegraphics[width=\linewidth]{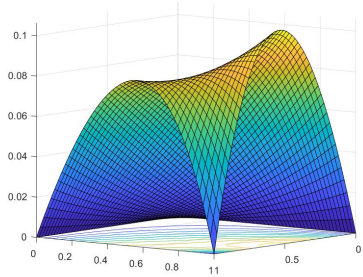}
					\caption{Graph of $B_{3,4}f-B_{3,4,2}f$ }\label{F4.1.2}
				\end{figure}
			\end{minipage}
		\end{minipage}

	\medskip
		
		For  $j=2$
		in Table \ref{table1} we present the error of approximation for AKR operator ($E_{AKR}f$) and Bernstein operator ($E_Bf$) for certain values of $n$ and $m$.
		Note that in this case the approximation by AKR  operator  is better than the approximation by the Bernstein operator. 
		
		\medskip
		
		\begin{tb}\label{table1}
			\centering
			{\it  Error of approximation }
			
			$  $
			
			\begin{tabular}{ccccccc}\hline
				$n=m$&$10$&$20$&$30$&$40$&$50$&$60$ \\  \hline
				$E_{B}f$   &  $0.1057$ &  $0.0516$ & $ 0.0342$ & $0.0255$ & $0.0204$ & $0.0169$\\
				$E_{AKR}f $   &  $0.0449$ &  $0.0215$ & $0.0142$ & $0.0106$ & $0.0084$ & $0.0070$\\
				\hline
			\end{tabular}
		\end{tb}
	\end{exam}
	
	\medskip
	
	Next we present two methods in order to construct a function $f\in C([0,1]^2)$ for which the   AKR operator acts better than Bernstein operator.

	\vskip 0.1in

	{\bf I.} Let $\omega(u,v)$ be such that $\omega(\cdot,v_0)$ and $\omega(u_0,\cdot)$ are increasing and convex, for all $u_0,\, v_0\in[0,1]$. For $x,y\in[0,1]$ let $a(x)$ and $b(y)$ be increasing and convex functions. Define $g(x,y)=\omega(a(x),b(y))$. It can be verified immediately  that $g(\cdot,y_0)$ and $g(x_0,\cdot)$ are increasing and convex, for all $x_0,\, y_0\in[0,1]$. Then, we can consider $f(x,y):=g(x^j,y^j)$ for which AKR operator acts better than Bernstein operator.
	
	\begin{exam} Let $\omega(u,v)=\tan\left(\dfrac{\pi}{4}uv\right),\,0\leq u,v\leq 1$, $a(x)=2^x-1$, $b(y)=y^3$. 
		Then, for $f(x,y)=\tan\left(\dfrac{\pi}{4}(2^{x^j}-1)y^{3j}\right)$ the approximation by AKR operator is better than the approximation by the Bernstein operator.
		
			For $n=3$, $m=4$ and $j=2$  Figures \ref{F4.2.1} and \ref{F4.2.2} illustrate the inequalities\linebreak
		$f\leq B_{3,4,2}f\leq B_{3,4}$.

		\begin{minipage}{\linewidth}
			\centering
			\begin{minipage}{0.45\linewidth}
				\begin{figure}[H]
					\includegraphics[width=\linewidth]{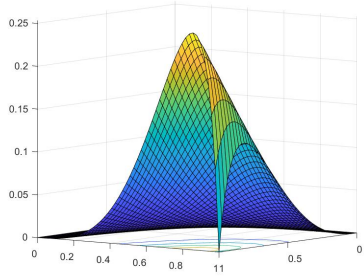}
					\caption{Graph of $B_{3,4,2}f-f$}\label{F4.2.1}
				\end{figure}
			\end{minipage}
			\hspace{0.05\linewidth}
			\begin{minipage}{0.45\linewidth}
				\begin{figure}[H]
					\includegraphics[width=\linewidth]{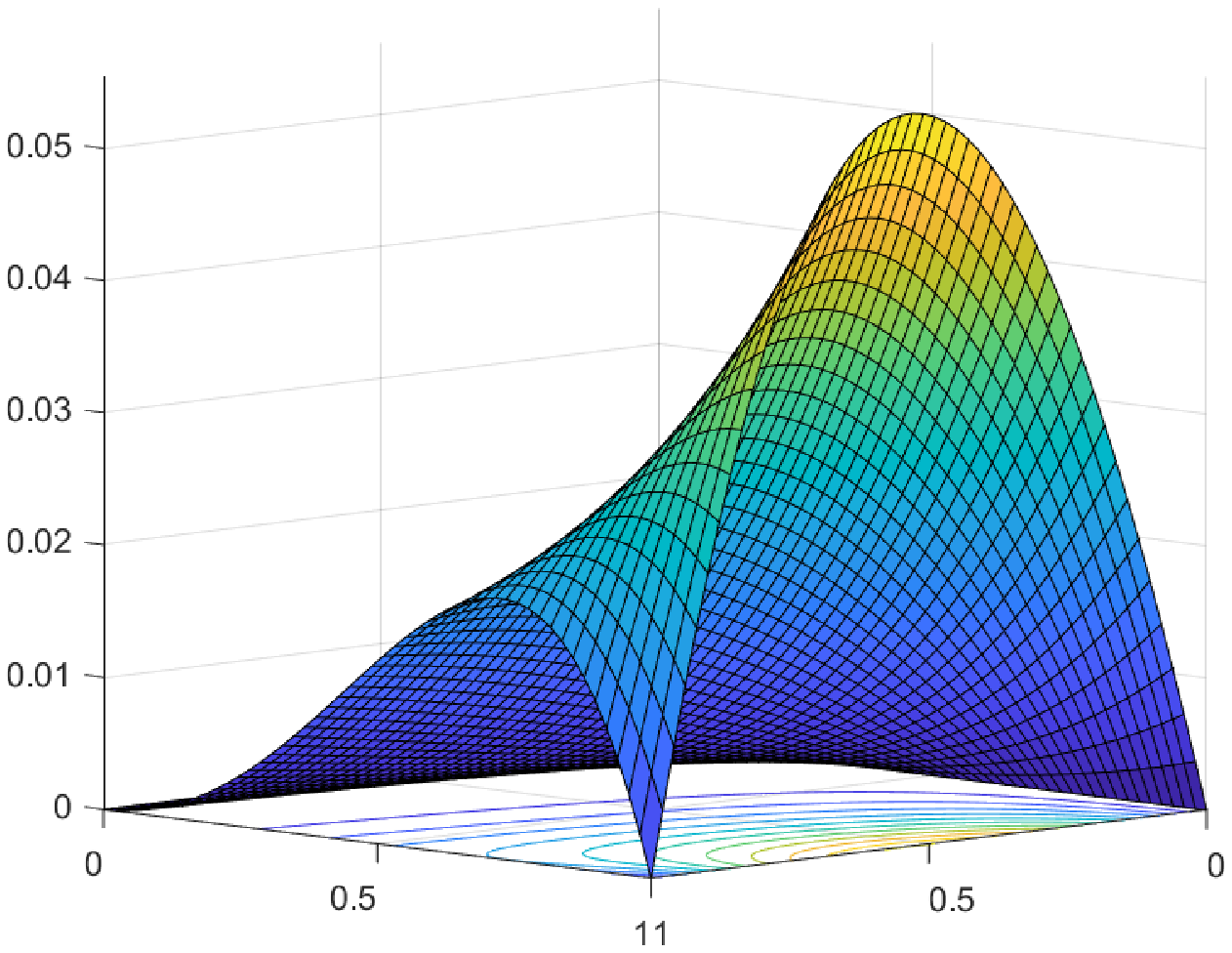}
					\caption{Graph of $B_{3,4}f-B_{3,4,2}f$ }\label{F4.2.2}
				\end{figure}
			\end{minipage}
		\end{minipage}
	\end{exam}
	\medskip
	
	{\bf II.} We want to dispose of two functions $\varphi$ and $\psi$ satisfying conditions (\ref{eB12}), (\ref{eB13}) and (\ref{eB10}). Choose $\varphi\in C^{1,2}([0,1]^2)$ such that
	
	\begin{equation}\label{eC1}\varphi\geq 0,\,\, \varphi^{\prime}_x\geq 0,\,\, \varphi^{\prime}_y\geq 0,\,\, y\varphi^{\prime\prime}_{y^2}-(j-1)\varphi_y^{\prime}\geq 0.\end{equation}
	Moreover, assume that the function \begin{equation}\label{eD2}\psi(x,y):=y^{1-j}\displaystyle\int_{0}^x t^{j-1}\varphi_y^{\prime}(t,y)dt\end{equation} is in $C^{1,1}([0,1]^2)$.
	It is easy to verify that the conditions (\ref{eB12}), (\ref{eB13}) and (\ref{eB10}) are indeed satisfied.
	
	As far as conditions (\ref{eC1}) are concerned we can start with a function $\tau\in C^{1,1}([0,1]^2)$ subject to the conditions
	\begin{equation}\label{eD3} \tau\geq 0, \,\, \tau^{\prime}_x\geq 0,\,\, \tau_y^{\prime}\geq 0,  \end{equation}
	and then construct \begin{equation}\label{eD4}\varphi(x,y)=\displaystyle\int_0^y s^{j-1}\tau(x,s)ds.\end{equation}
	It will satisfy (\ref{eC1}). From (\ref{eD2}) we get
	\begin{equation}
		\label{eD5}
		\psi(x,y)=\displaystyle\int_0^x t^{j-1}\tau(t,y)dt.
	\end{equation}

\begin{conclusion}Starting with $\tau$ satisfying (\ref{eD3}), the equations (\ref{eD4}), (\ref{eD5}), (\ref{eB11}) yield the function  
\begin{equation}
	\label{eD6}
	f(x,y)=\displaystyle\int_0^x\int_0^yt^{j-1}s^{j-1}\tau(t,s)dtds,
	\end{equation}
which is approximated by AKR operator better than by Bernstein operator, in the sense  that 
\begin{equation}\label{eE1}f\leq B_{n,m,j}f\leq B_{n,m}f. \end{equation}
\end{conclusion}
  In Example \ref{Ex4}, $\tau(x,y)=4(1+x^2y^2)e^{x^2y^2}.$
	
	\begin{exam} Let $\tau(t,s)=\sin \dfrac{\pi(t+s)}{4}$. The function $f$ given by (\ref{eD6}) is 
		\begin{align*} f(x,y)&=\dfrac{1}{\pi^4}\left[-64\pi(y + x)\cos(\pi(y + x)/4) + (-16\pi^2xy + 256)\sin(\pi(y + x)/4) \right.\\
		&+ \left.64\pi x\cos(\pi x/4) + 64\pi y\cos(\pi y/4) - 256\sin(\pi x/4) - 256\sin(\pi y/4)\right]. \end{align*}
	Figures \ref{F4.3.1} and \ref{F4.3.2} show that $f\leq B_{4,4,2}f\leq B_{4,4}f$. This is an illustration of the inequalities (\ref{eE1}).

\begin{minipage}{\linewidth}
	\centering
	\begin{minipage}{0.45\linewidth}
		\begin{figure}[H]
			\includegraphics[width=\linewidth]{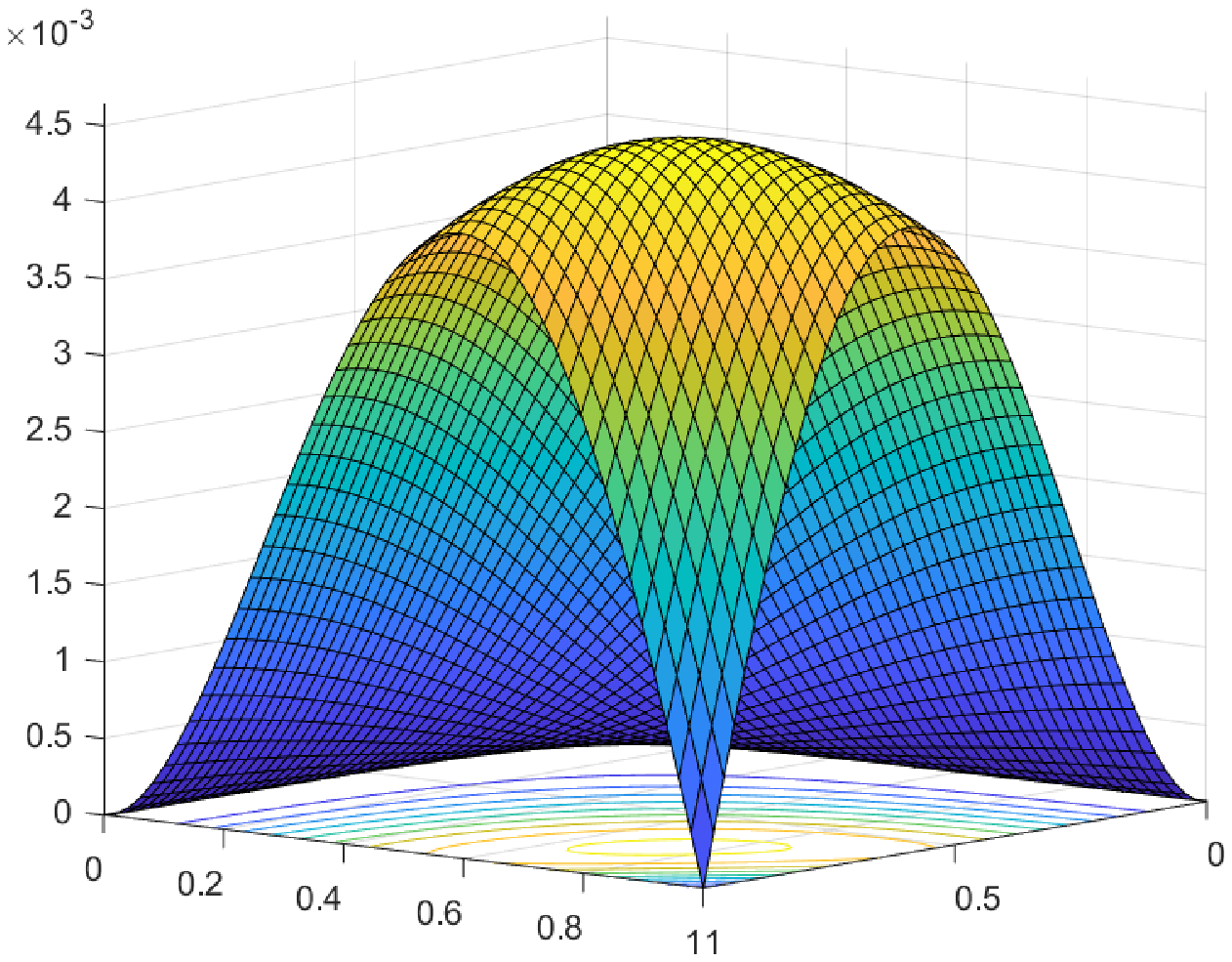}
			\caption{Graph of $B_{4,4,2}f-f$}\label{F4.3.1}
		\end{figure}
	\end{minipage}
	\hspace{0.05\linewidth}
	\begin{minipage}{0.45\linewidth}
		\begin{figure}[H]
			\includegraphics[width=\linewidth]{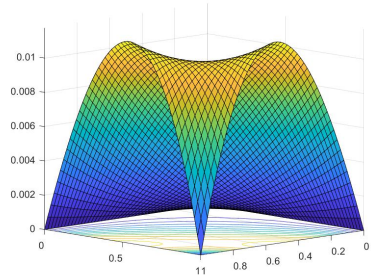}
			\caption{Graph of $B_{4,4}f-B_{4,4,2}f$ }\label{F4.3.2}
		\end{figure}
	\end{minipage}
\end{minipage}
	\end{exam}

	\subsection{On computing with Bernstein and AKR operators}
	Here we briefly describe some implementation details, which are useful for people interested in reproducing the results of the paper and/or wish to get more computational insights. In GitHub at the link
	
	\url{https://github.com/demarchi17/Bernstein-and-AKR-Operators} 
	
	\noindent there is the Matlab script {\tt AldazBernstein1d2d.m} that allows to construct both operators on $I=[0,1]$ and in the unit square $Q=[0,1] \times [0,1]$. In the
	interval we have considered the function
	$f(x)=-\dfrac{2}{\pi}x\cos\dfrac{\pi x}{2}+\dfrac{4}{\pi^2}\sin\dfrac{\pi x}{2}$ (see Example \ref{ex.1}) while in the
	unit square the 3 functions $f(x,y)$ of Examples 4.2, 4.3 and 4.4. 
	
	On input, the user is asked to provide
	the higher degree $N$ up to which to approximate the functions and the AKR index $j$. Then, a loop on the degree $n$ gives the results presented in Figures \ref{fig:1},\ref{fig:2},..., \ref{F4.3.1},\ref{F4.3.2} (see the corresponding Examples for details). More in details, the script allows: 
	\begin{itemize}
		\item to compute Bernstein and AKR points in $I$ and $Q$;
		\item to construct Bernstein $B_n$ and AKR operators $B_{n,m,j}$ and evaluate them on a suitable and larger grid, say $X_M$ where $M=(n+1)^2$ (constructed via {\tt meshgrid} on the unit square);
		\item in $[0,1]$ to compute the relative 2-norm error 
		$$E_r:=\frac{\| f- B\|_2}{\|f\|_2 }$$
		where $B$ is one of the operators and $f$ one of the functions previously considered;
		\item in $[0,1]^2$ to compute the errors
			$B_{n,m,j}f-f$ and
		$B_{n,m} f - B_{n,m,j}f$
		\item to make the plots.
	\end{itemize}
	The script available at GitHub can be downloadable and the interested readers can play and tell us if there are improvements and bugs.

	\section{Conclusions and further work}
	The AKR operators have been subject of intense research. They are important in Approximation Theory, where rate of convergence and Voronovskaja formula play a significant role. Their shape preserving properties are useful in CAGD. 
	
	Our paper has two aims. On one hand, we introduce the bivariate version of the AKR operators on $C([0,1]^2)$ and investigate some approximation properties of them. On the other hand we compare, in the univariate case and also in the bivariate case, the approximation provided by AKR operators with that provided by Bernstein operators, in the spirit of \cite{Aldaz2}, \cite{Aldaz1}. More precisely, we describe families of functions which are better approximated by AKR operators and families of functions which are better approximated by Bernstein operators. Numerical and graphical experiments illustrate the theoretical results.
	
	Theorem 22 and Theorem 24 in \cite{AKR2009} exhibit shape preserving properties of AKR operators in the univariate setting. See also \cite{Aldaz2}, \cite{Aldaz1}.
	It is known that the usual convexity is not generally invariant under the bivariate Bernstein operators. Families of convex functions for which the convexity is preserved by these operators are described in \cite{70} and \cite[Section 3.4]{Altomare_Rasa}. The monotonicity of the sequence $(B_{n,j}f)_{n\geq 1}$ for generalized convex functions $f$ is presented in \cite[Theorem 19]{AKR2009}. In the multivariate case this kind of monotonicity is investigated in \cite[Section 3.5]{Altomare_Rasa}.
The preservation of convexity and the above mentioned kind of monotonicity under the bivariate AKR operators deserve to be investigated, together with their applications to CAGD.
	
	$  $
	
	\noindent{\bf Acknowledgments}. This work has been accomplished within the Rete Italiana di Approssimazione and the UMI Group "Teoria dell'Approssimazione e Applicazioni". The first author has been supported by the INdAM-GNCS {\it Visiting Professors program 2021}. The second author has been also partially supported by the Erasmus+ Programme for Teaching Staff Mobility between the Universities of Padova and Sibiu.
	
	\bibliographystyle{amsplain}
	
\end{document}